\documentclass[a4paper]{article}
\usepackage{authblk}
\usepackage{cite}
\usepackage{amsthm,amssymb,amsmath,mathrsfs, siunitx}
\usepackage{fancyhdr}
\usepackage{enumerate}
\usepackage{verbatim}
\usepackage{pstricks}
\usepackage{graphicx}
\usepackage{placeins}
\usepackage{tikz}
\usepackage{mathtools}
\usepackage{caption}
\usepackage[top=1.5cm, bottom=1.5cm, left=1.2cm, right=1.3cm]{geometry}
\newtheorem{theorem}{Theorem}[section]
\newtheorem{defi}[theorem]{Definition}
\newtheorem{lemma}[theorem]{Lemma}
\newtheorem{example}[theorem]{Example}
\newtheorem{prop}[theorem]{Proposition}

\newtheorem{corol}[theorem]{Corollary}
\newtheorem{remark}[theorem]{Remark}

\newtheorem*{que}{Question}

\begin{document}
	
	\title{Recognizability of Being Point Determining}
	\author[1]{M. H. Shirdareh Haghighi\thanks{Corresponding author. Email: \texttt{shirdareh@shirazu.ac.ir}}}
	\author[1]{Asma Namazi\thanks{Email: \texttt{asma.namazi@outlook.com}}}
	\author[1]{Zahra Rahimi\thanks{Email: \texttt{zahra.paliz.rahimi@gmail.com}}}
	\author[1]{A. M. Ghazanfari\thanks{Email: \texttt{amir.m.ghazanfari@gmail.com}}}
	
	\affil[1]{Department of Mathematics and Computer Science, Shiraz University, Shiraz, Iran}
	\date{\today}
	\maketitle
\begin{abstract}
	A graph $G$ is point determining, if no two vertices have the same neighborhoods. In this paper, we show that this property is recognizable from the deck of cards of a graph.
	
	
\end{abstract}

\section{Introduction and Preliminaries}
In this paper all graphs are finite and simple, i.e. without any loops and multiple edges. The graph theoric notations and definitions are as in \cite{bondy2007}.

The graph reconstruction problem, introduced independently by Kelly \cite{kelly1957} and Ulam \cite{ulam1960}, is one of the most studied open problems in graph theory. Given a graph $G$, the \textit{deck} of $G$ is the multiset of all unlabeled vertex-deleted subgraphs (cards) of $G$. The Reconstruction Conjecture states that every graph with at least three vertices is uniquely determined, up to isomorphism, by its deck of cards. Given a deck $D$ of cards, any graph $H$ whose deck is $D$ is referred to as a \textit{solution} to this deck.

Over the years, reconstruction theory has developed into a rich area of research concerned with determining which graph classes, properties, and invariants can be recovered from the deck of a graph. As stated in \cite{chartrand2015}, a graphical parameter or property is said to be \textit{recognizable} (\textit{reconstructible}) if it can be determined from the deck of cards of a graph. Among the elementary examples of recognizable properties are the order, size, and degree sequence of a graph. A more general result is provided by Kelly's Lemma \cite{kelly1957}, which states that the number of subgraphs of $G$ isomorphic to a fixed graph $H$, where $|V(H)|<|V(G)|$, is reconstructible from the deck. This immediately yields the reconstructibility of several subgraph-counting parameters, including the number of triangles. More sophisticated recognizable properties include the chromatic polynomial, the Tutte polynomial, and the characteristic polynomial \cite{bondy1991}, \cite{tutte1979}.

Among the graph classes arising naturally in structural graph theory are the point determining graphs introduced by Sumner \cite{sumner1973}. A graph is point determining if no two distinct vertices have the same neighborhoods. Point determining graphs eliminate the presence of false twins, that are pairs of non-adjacent vertices with the same neighborhoods, and therefore represent graphs whose vertices are uniquely determined by their neighborhoods. Point determining graphs arise naturally in the study of graph symmetries, graph realizations, and graph homomorphisms. They have been investigated extensively by some authors including Sumner\cite{sumner1973} and Geoffroy \cite{geoffroy1978}, leading to a variety of structural and enumerative results.

The close connection between neighborhood structure and reconstruction problem suggests that point determining graphs should be a natural object of study within reconstruction theory. Vertices with identical neighborhoods often create ambiguities in vertex-deleted subgraphs, whereas, point determining graphs seem to reduce these ambiguities. Consequently, questions concerning the recognizability and reconstructibility of being point determining and associated properties arise naturally in this context.

\begin{defi}
Let $G$ be a graph. The neighborhood of a vertex $x$ in $G$, denoted by $N_G(x)$, is the set of vertices that are adjacent to $x$. Two vertices $x$ and $y$ are false twins if $N_G(x)=N_G(y)$. A graph is called \textit{point determining (PD)} if no two vertices have the same neighborhoods.
\end{defi}

Note that a pair of false twins are non-adjacent.
The following theorem about point determining graphs is very crucial.

\begin{theorem}\label{nucleus}
	\cite[Theorem 2]{sumner1973} Let $G$ be a nontrivial point determining graph. Then there exists a vertex $x \in G$ such that $G-x$ is point determining. \hfill $\square$
\end{theorem}
 
 Such a vertex $x$, in Theorem \ref{nucleus}, is called a nucleus vertex of $G$ and the set of nucleus vertices of $G$ is denoted by $G^\circ$.

In Section \ref{G01} we characterize PD graphs $G$ with a single nucleus vertex and discuss PD graphs $G$ with $|G^\circ|=2$. In Section \ref{main} we prove our main result that is reconizability of being point determining. We conclude this paper by a relevant question in Section \ref{conclusion}.

\section{PD Graphs with at Most Two Nucleus Vertices}\label{G01}

In this section, first we give a full characterization of PD graphs with $|G^\circ|=1$ and then give properties of PD graphs with $|G^\circ|=2$ which are needed in this paper.

For $n\geq 1$ we denote by $\mathcal{A}_n$ the bipartite graph with parts $\{u_1,\ldots,u_n\}$ and $\{v_1,\ldots,v_n\}$, such that $u_i$ is adjacent to $v_j$ if $i \geq j$, $1\leq i,j \leq n$. See Figure \ref{An}. Clearly, $\mathcal{A}_n$ satisfies the following properties.

\begin{itemize}
	\item[(\AA1)] $N_{\mathcal{A}_n}(v_i )=N_{\mathcal{A}_n}(v_{i-1} )-\{u_{i-1}\}$, for $2\leq i\leq n$,
	\item[(\AA2)] $N_{\mathcal{A}_n}(u_i )=N_{\mathcal{A}_n}(u_{i+1} )-\{v_{i+1}\}$, for $1 \leq i \leq n-1$,
	\item[(\AA3)] $\mathcal{A}_n$ is PD with $\mathcal{A}_n^\circ =\{v_1 ,u_n\}$,
	\item[(\AA4)] $\mathcal{A}_n-\{u_n , v_n\} = \mathcal{A}_{n-1}$.
\end{itemize} 
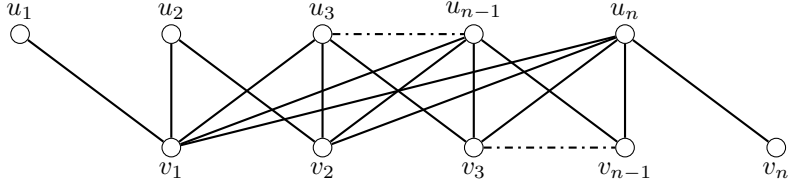
\begin{figure}[h]
\begin{center}
\begin{tikzpicture}[
	vertex/.style={
		circle,
		draw,
		fill=white,
		inner sep=0pt,
		minimum size=7pt
	},
	edge/.style={
		line width=0.8pt
	},
	dashededge/.style={
		line width=0.8pt,
		dash pattern=on 3pt off 2pt on 1pt off 2pt
	}
	]
	
	
	\node[vertex] (u1) at (0,1.5) {};
	\node[vertex] (u2) at (2,1.5) {};
	\node[vertex] (u3) at (4,1.5) {};
	\node[vertex] (un1) at (6,1.5) {};
	\node[vertex] (un) at (8,1.5) {};
	
	\node[vertex] (v1) at (2,0) {};
	\node[vertex] (v2) at (4,0) {};
	\node[vertex] (v3) at (6,0) {};
	\node[vertex] (vn1) at (8,0) {};
	\node[vertex] (vn) at (10,0) {};
	
	
	\node[above=2pt] at (u1) {$u_1$};
	\node[above=2pt] at (u2) {$u_2$};
	\node[above=2pt] at (u3) {$u_3$};
	\node[above=2pt] at (un1) {$u_{n-1}$};
	\node[above=2pt] at (un) {$u_n$};
	
	\node[below=2pt] at (v1) {$v_1$};
	\node[below=2pt] at (v2) {$v_2$};
	\node[below=2pt] at (v3) {$v_3$};
	\node[below=2pt] at (vn1) {$v_{n-1}$};
	\node[below=2pt] at (vn) {$v_n$};
	
	
	\draw[edge] (u1)--(v1);
	
	\draw[edge] (u2)--(v1);
	\draw[edge] (u2)--(v2);
	
	\draw[edge] (u3)--(v1);
	\draw[edge] (u3)--(v2);
	\draw[edge] (u3)--(v3);
	
	\draw[edge] (v1)--(un);
	\draw[edge] (v2)--(un);
	\draw[edge] (v3)--(un);
	\draw[edge] (un1)--(vn1);
	\draw[edge] (un)--(vn1);
	\draw[edge] (un)--(vn);
	\draw[edge] (un1)--(v1);
	\draw[edge] (un1)--(v2);
	\draw[edge] (un1)--(v3);
	
	\draw[dashededge] (u3)--(un1);
	\draw[dashededge] (v3)--(vn1);
	
\end{tikzpicture}
\caption{The graph $\mathcal{A}_n$}\label{An}
\end{center}
\end{figure}



\begin{prop}
Let $k_1,\cdots,k_m$ be positive integers. Then the graph $G=\mathcal{A}_{k_1}+\mathcal{A}_{k_2}+\cdots+\mathcal{A}_{k_m}+K_1$ is a PD graph with $G^\circ=V(K_1)$.
\end{prop}

\begin{proof}
First note that the disjoint union of some PD graphs with at most one isolated vertex in the union, is PD. By (\AA1) and (\AA2) all non-nucleus vertices of $\mathcal{A}_{k_i}$, $1 \leq i \leq m$, are still non-nucleus in $G$. Let $ n \in \{k_1 , \ldots, k_m\}$ and consider $\mathcal{A}_n$ with vertices  $\{u_1 , \ldots, u_{n}, v_1 , \ldots , v_{n}\}$ (according to the definition). By deletion of $v_1$, the vertex $u_1$ is the false twin of the isolated vertex. Similarly, by deletion of $u_n$, the vertex $v_n$ is the false twin of the isolated vertex. Hence, any nucleus vertex of $\mathcal{A}_{n}$, is not nucleus in $G$. Clearly, the isolated vertex is in $G^\circ$ and we have $G^\circ=V(K_1)$.
\end{proof}																																																										

The converse of this proposition is also true. First we need more results from \cite{sumner1973}.
\begin{prop}\label{noisole}
	\cite[Corollary 1]{sumner1973} Let $G$ be a point determining graph that has no isolated vertices, then $|G^\circ|\geq 2$.
\end{prop}	
					
\begin{lemma}\label{birooni}
	\cite[Lemma 6 and Lemma 7]{sumner1973} Let $G$ be a point determining graph and $a , b, c \in G$ are distinct vertices with $N_G(a)=N_G(b)-\{c\}$. Then the following statements hold.
	\begin{itemize}
		\item[(i)] If $a \not\in G^\circ$, then there exists $d \in G-\{a,c\}$ with $ N_G(c)=N_G(d)-\{a\}$.
		\item[(ii)] If $b \not\in G^\circ$, then there exists $d \in G-\{b,c\}$ with $ N_G(d)=N_G(c)-\{b\}$.
	\end{itemize}
\end{lemma}
\begin{lemma}\label{notejtema}
Let $G$ be a disconnected point determining graph with non-trivial components. Then $G^\circ = \bigcup_{H} H^\circ$, where $H$ runs over all connected components of $G$. \hfill$\square$
\end{lemma}																																												
\begin{prop}\label{G0=1}
	Let $G$ be a point determining graph with $G^\circ=\{x_0\}$. Then there exist positive integers $k_1,\cdots,k_m$ such that $G=\mathcal{A}_{k_1}+\mathcal{A}_{k_2}+\cdots+\mathcal{A}_{k_m}+x_0$.
\end{prop}
\begin{proof}
	By Proposition \ref{noisole}, if $|G^\circ|=1$ and $G$ is PD, then the unique isolated vertex of $G$ is $x_0$. Therefore, we may assume that $G= H+ K_1$ where $V(K_1)=\{x_0\}$ and $H$ is PD with no isolated vertices. Let $x \in H^\circ$. Since $G-x$ is not PD, then there exist vertices $p,q\in G-x$ such that $N_G(p)=N_G(q)-\{x\}$. Both $p$ and $q$ cannot be in $H$ at the same time, because $x \in H^\circ$. It follows that $p=x_0$, which means $q$ is a pendant vertex adjacent to $x$. So,
	\begin{itemize}
		\item[($*$)]  every vertex which is nucleus in $H$, is adjacent to a pendant vertex in $H$.
	\end{itemize}

	Let $H_1$ be a connected component of $H$ and $v_1 \in H_1^\circ$. By Lemma \ref{notejtema},  $v_1 \in H^\circ$. By ($*$), $v_1$ is adjacent to a pendant vertex $u_1$. If $u_1 \in H_1^\circ$, by Lemma \ref{notejtema}, $u_1 \in H^\circ$ . By ($*$), we deduce that $u_1$ has a pendant neighbor. Consequently, $H_1= K_2=\mathcal{A}_1$. If $u_1 \not\in H_1^\circ$, then there exist $v_2 , a \in G-u_1$ such that $N_G(v_2)=N_G(a)-\{u_1\}$. If $H_1=K_2$ then $u_1 \in H_1^\circ$, which is a contradiction. Therefore, $H_1\neq K_2$ and $v_2\neq x_0$. Since $u_1$ is pendant in $G$, we have $a=v_1$. Since $N_G(v_{2})=N_G(v_1)-\{u_1\}$, the degree of any vertex adjacent to $v_{2}$ is at least two. By ($*$) and Lemma \ref{notejtema}, we conclude that $v_{2} \notin H_1^\circ$. Therefore, there exist vertices $b , u_2 \in G-v_2$ such that $N_G(b) =N_G(u_2)-\{v_2\}$. Since $N_G(v_2)=N_G(v_1)-\{u_1\}$ and $u_2$ is adjacent to $v_2$, the vertex $u_2$ is adjacent to $v_1$. So, $b$ is adjacent to $v_1$ and not to $v_2$. Hence, $b=u_1$, $N_G(u_1) =N_G(u_2)-\{v_2\}$, $\deg_G(u_2)=2$ and $\mathcal{A}_2$ is an induced subgraph of $H_1$.

	Generally, for $i\geq 2$, suppose $\mathcal{A}_i$ is an induced subgraph of $H_1$, where $V(\mathcal{A}_i)=\{u_1, \ldots, u_i , v_1 , \ldots, v_i\}$ and
	\begin{itemize}
		\item[(i)] $N_{G}(v_k )=N_{G}(v_{k-1} )-\{u_{k-1}\}$, for $2\leq k\leq i$,
		\item[(ii)] $N_{G}(u_k )=N_{G}(u_{k+1} )-\{v_{k+1}\}$, for $1 \leq k \leq i-1$,
		\item[(iii)] $\deg_G(u_k)=k$, for $1 \leq k \leq i$.
	\end{itemize}

	If $u_{i}\in H_1^\circ$, we have $u_{i} \in H^\circ$ by Lemma \ref{notejtema}. By ($*$), $u_{i}$ is adjacent to a pendant vertex in $H_1$. We have $N_G(u_{i})=\{v_1 , \dots, v_{i}\}$ and $\{u_t, \dots , u_{i}\} \subseteq N_{H_1}(v_t)$ for $1 \leq t\leq i$. Then the only valid candidate for this pendant vertex is $v_{i}$. Summing up, $v_{i}$ is a pendant vertex adjacent to $u_{i}$ and $H_1 =\mathcal{A}_{i}$.

If $u_i \notin H_1^\circ$, then, there exist vertices $ v_{i+1},v \in H_1-u_i$ such that $N_{H_1}(v_{i+1})=N_{H_1}(v)-\{u_i\}$. Since $H_1$ is a connected component of $G$, hence $N_{G}(v_{i+1})=N_{G}(v)-\{u_i\}$.  By (ii) and (iii), $N_{G}(u_i)=\{v_1 , \dots, v_i\}$. So, $v=v_j$ for some $j$, $1 \leq j \leq i$.  For $ 1\leq t\leq i$, we know that $N_{\mathcal{A}_i}(v_t)=\{u_t , \dots, u_i\}$, thus $v_{i+1}\neq v_t$. Also, $N_{\mathcal{A}_i}(u_t)=\{v_1 , \ldots, v_t\}$ for $ 1\leq t\leq i$. If $j\neq i$, then $N_G(v_{i+1})=N_G(v_j) -\{u_i\}$ for some $j$, $1\leq j<i$. Thus, $\{v_1 , \dots ,v_j, v_{i+1}\} \subseteq N_{G}(u_j)$ and consequently $|N_{G}(u_j)| > j$, which contradicts (iii). Therefore, $j=i$ and $v=v_i$ (statement (i) for $k=i+1$). Using $N_G(v_{i+1})=N_G(v_i)-\{u_i\}$ and (i), inductively, we have $N_G(v_{i+1})=N_G(v_1)-\{u_1, \ldots, u_i\}$. Thus, $v_{i+1} \notin \{u_1, \ldots, u_n\}$. Since $N_G(v_{i+1})=N_G(v_i)-\{u_i\}$, the degree of any vertex adjacent to $v_{i+1}$ is at least two. By ($*$) and Lemma \ref{notejtema}, we conclude that $v_{i+1} \notin H_1^\circ$. Therefore, there exist vertices $s , u_{i+1} \in G-v_{i+1}$ such that $N_G(s) =N_G(u_{i+1})-\{v_{i+1}\}$. Since $N_G(v_{i+1})=N_G(v_i)-\{u_i\}$ and $u_{i+1}$ is adjacent to $v_{i+1}$, $u_{i+1}$ is adjacent to $v_i$. So, $s$ is adjacent to $v_i$ and not to $v_{i+1}$. Hence, $s=u_i$, $N_G(u_i) =N_G(u_{i+1})-\{v_{i+1}\}$ (statement (ii) for $k=i$), $\deg_G(u_{i+1})=i+1$ (statement (iii) for $k=i+1$) and $\mathcal{A}_{i+1}$ is an induced subgraph of $H_1$.

Continuing this way, we conclude that $H_1=\mathcal{A}_n$ for some $n\geq1$ and the proposition follows.  
\end{proof}
Now we study connected PD graphs $G$ with $|G^\circ |=2$. For a graph $G$ and $z \in V(G)$ we define the subgraph $N_z$
to be the induced subgraph by $N_G(z)$ in $G$. We need a theorem from \cite{sumner1973}.

\begin{theorem}\label{xya}
	\cite[Theorem 6 and Corollary 2]{sumner1973} If G is a connected, point determining graph with $G^\circ = \{x,y\}$, then $x$ and $y$ are adjacent and every vertex of $G$ is adjacent to exactly one of $x$ and $y$.
\end{theorem}

\begin{lemma} \label{nKH}
	Let $G$ be a connected point determining graph with $G^\circ = \{x,y\}$. Then there exists $\alpha \geq 1$ such that either $N_x=\alpha K_1$ or $N_x = \alpha K_1 +H$ for some point determining graph $H$.
\end{lemma}

\begin{proof}
	By Theorem \ref{xya}, $y \in N_x$ and no vertex in $N_x -\{{y}\}$ is adjacent to $y$. Therefore, $y$ is an isolated vertex of $N_x$, and we can assume $N_x = \alpha K_1 + H$ such that $H$ has no isolated vertices in $N_x$ , $\alpha \geq 1$.  If $H = \emptyset$, then there is nothing to prove. So, let $H \neq \emptyset$. In contrary, suppose that $H$ is not PD.  Among all vertices in $H$ that have false twins, consider  $z$ with largest degree in $G$. Since $z \in N_x$  and $y$ is an isolated vertex in $N_x$, $z\neq x ,y$. Thus, $G-z$ is not PD and there exist vertices $p,q \in G-z$ such that $N_G(p)=N_G(q)-\{z\}$. By Theorem \ref{xya}, $N_x$ and $N_y$ are disjoint and every vertex in $G$ is either adjacent to $x$ or $y$, but not both. Hence the vertices $p$ and $q$ both are either in $N_x$ or in $N_y$ . If $p,q \in N_y$, since $N_G(p)=N_G(q)-\{z\}$, we deduce that $p$ and $z$ are not adjacent. But we know that $z$ and $x$ are adjacent and it follows that $p \neq x$. By Lemma \ref{birooni}, there exists $r \in G-\{p , z\}$ such that $N_G(z)=N_G(r)-\{p\}$. Since $z \in N_x$, we have $r \in N_x$. Also, $N_H(z)=N_H(r)$, because $p \in N_y$. This means $r$ is a false twin of $z$ with $\deg_G(z)< \deg_G(r)$, which contradicts the choice of $z$. Therefore, $p,q \in N_x$. Restrict the equality $N_G(p)=N_G(q)-\{z\}$ to the subgraph $N_x$. Consider a false twin $z'$ of $z$ in $N_x$. Since $p$ is not adjacent to $z$, it is not adjacent to $z'$, too ; and since $q$ is adjacent to $z$, it is also adjacent to $z'$, which contradicts the equality $N_G(p)=N_G(q)-\{z\}$. Consequently, $H$ has no pair of false twins and is a PD graph.
\end{proof}

Using the same reasoning as in the proof of Proposition \ref{G0=1}, together with Lemma \ref{nKH}, we obtain the following result.

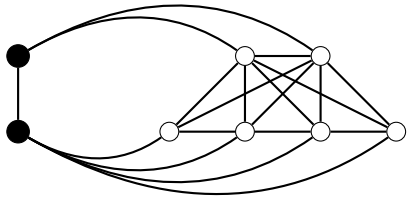
\begin{figure}[h]
	\begin{center}
		\begin{tikzpicture}[
			vertex/.style={circle, draw, fill=white, inner sep=2.5pt},
			blackvertex/.style={circle, draw, fill=black, inner sep=3pt},
			edge/.style={thick}
			]
			
			\node[blackvertex] (a) at (0,1) {};
			\node[blackvertex] (b) at (0,0) {};
			
			\node[vertex] (c) at (3,1) {};
			\node[vertex] (d) at (4.0,1) {};
			\node[vertex] (e) at (2.0,0) {};
			\node[vertex] (f) at (3,0) {};
			\node[vertex] (g) at (4.0,0) {};
			\node[vertex] (h) at (5,0) {};
			
			\draw[edge] (a)--(b);
			
			\draw[edge] (e)--(c);
			\draw[edge] (e)--(f);
			\draw[edge] (c)--(f);
			\draw[edge] (c)--(d);
			\draw[edge] (f)--(g);
			\draw[edge] (d)--(g);
			\draw[edge] (d)--(h);
			\draw[edge] (g)--(h);
			
			\draw[edge] (c)--(g);
			\draw[edge] (f)--(d);
			\draw[edge] (c)--(h);
			\draw[edge] (d)--(e);
			
			\draw[edge] (a) to[out=30,in=145] (d);
			\draw[edge] (a) to[out=30,in=145] (c);
			
			\draw[edge] (b) to[out=-30,in=-145] (g);
			\draw[edge] (b) to[out=-30,in=-145] (e);
			\draw[edge] (b) to[out=-30,in=-145] (h);
			\draw[edge] (b) to[out=-30,in=-145] (f);

		\end{tikzpicture}
		\caption{A graph with double nucleus}
	\end{center}
\end{figure}

\begin{lemma}\label{n1}
	Let $G$ be a connected point determining graph with $G^\circ = \{x,y\}$. Then there exists $\alpha \geq 1$ such that either $N_x = \alpha K_1$ or $N_x = \alpha K_1 +\mathcal{A}_{k_1}+\mathcal{A}_{k_2}+\cdots+\mathcal{A}_{k_m}$ for some positive integers $k_1,\ldots,k_m$.
\end{lemma}
\begin{proof}
	By Lemma \ref{nKH}, there exists $\alpha \geq 1$ such that $N_x=\alpha K_1 +H$, where $H$ has no isolated vertices and is PD whenever $H\neq \emptyset$. If $H = \emptyset$, then there is nothing to prove. Let $H \neq \emptyset$ and $z \in H^\circ$. Then $z \neq y$, because $y \in \alpha K_1$, by Lemma \ref{nKH}. Hence $G-z$ is not PD. Therefore, there exist vertices $p , q \in G-z$ such that $N_G(p)=N_G(q)-\{z\}$. If both $p$ and $q$ belong to $N_y$, then $p\neq x$, because $z$ is adjacent to $x$. By Lemma \ref{birooni}, there exists a vertex $r \in G-\{p,z\}$ such that $N_G(z)=N_G(r)-\{p\}$. The vertex $r$ is adjacent to $x$, because $z$ is adjacent to $x$; this means $r \in N_x$ and $r \not\in \alpha K_1$. By restricting the equality $N_G(z)=N_G(r)-\{p\}$ to the subgraph $H$, we have $N_{H}(z)=N_{H}(r)$. The vertex $r$ is a false twin of $z$ in $H$, and this contradicts the PD property of $H$. It follows that $p,q \in N_x$. Since $z \in H^\circ$ and $N_G(p)=N_G(q)-\{z\}$, then either $p$ or $q$ is not in $H$. Since $q$ is adjacent to $z$, then $p \not\in H$ and is an isolated vertex of $N_x$. We also deduce that $q$ is a pendant vertex in $H$. If $p\neq y$, by Lemma \ref{birooni}, there exists a vertex $u \in G-\{p,z\}$ such that $N_G(z)=N_G(u)-\{p\}$. The vertex $u$ is adjacent to $x$, because $z$ is adjacent to $x$; this means $u \in N_x$ and $u \in H$. By restricting the equality $N_G(z)=N_G(u)-\{p\}$ to the subgraph $H$, we have $N_{H}(z)=N_{H}(u)$. Thus, the vertex $u$ is a false twin of $z$ in $H$, and this contradicts the PD property of $H$. So,
	\begin{itemize}
		\item[($*$)]  every vertex $z \in H^\circ$, is adjacent to a pendant vertex $q \in H$ such that $N_G(y)=N_G(q)-\{z\}$.
	\end{itemize}  
	Let $H_1$ be a connected component of $H$. Since by Lemma \ref{nKH}, $H$ is PD, also $H_1$ is PD. For $v_1 \in H_1^\circ$, we have $v_1 \in H^\circ$, by Lemma \ref{notejtema}. Now ($*$) implies that $v_1$ is adjacent to a pendant vertex $u_1 \in H_1$  and $N_G(y)=N_G(u_1)-\{v_1\}$. If $u_1 \in H_1^\circ$, we have $u_1 \in H^\circ$ by Lemma \ref{notejtema}. Again ($*$) implies that $u_1$ has a pendant neighbor. Consequently, $H_1= K_2=\mathcal{A}_1$ and $N_G(y)=N_G(v_1)-\{u_1\}$.
	If $u_1 \notin H_1^\circ$, the subgraph $G-u_1$ is not PD. Then there exist vertices $ v_{2},a \in G-u_1$ such that $N_G(v_{2})=N_G(a)-\{u_1\}$. If both $a$ and $ v_{2}$ belong to $N_y$, then $v_{2}\neq x$. Therefore, by Lemma \ref{birooni}, there exists a vertex $b \in G-\{v_{2} , u_1\}$ such that $N_G(u_1)=N_G(b)-\{v_{2}\}$.  The vertex $b$ is adjacent to $x$, because $u_1$ is; this means $b \in N_x$. By restricting the equality $N_G(u_1)=N_G(b)-\{v_{2}\}$ to the subgraph $N_x$, we get $N_{N_x}(u_1)=N_{N_x}(b)$. Since $H_1$ is connected, $b$ is a false twin of $u_1$ in $H_1$, and this contradicts the PD property of $H_1$. Therefore, $v_{2},a \in N_x$ and more specifically $v_{2},a \in H_1$. Since every vertex adjacent to $v_{i+1}$ is also adjacent to $v_1$, no pendant vertex in $H_1$ is adjacent to $v_{2}$. Thus,  the vertex $v_{2}$ cannot be a nucleus vertex of $H_1$. Since $N_G(v_{2})=N_G(v_1)-\{u_1\}$ and $H_1 \neq \mathcal{A}_1$, then $v_{2} \neq y$. Therefore, $G-v_{2}$ is not PD and there exist $c , u_{2} \in G - v_{2}$ such that $N_G(c)=N_G(u_{2})-\{v_{2}\}$. If both $c$ and $u_{2}$ belong to $N_y$, then $c \neq x$ and by Lemma \ref{birooni} there exists a vertex $d \in G-\{c , v_{2}\}$ such that $N_G(v_{2})=N_G(d)-\{c\}$; So $d \in N_x$ and the restriction of the equality $N_G(v_{2})=N_G(d)-\{c\}$ to $N_x$ results in $N_{N_x}(v_{2})=N_{N_x}(d)$.  Since $H_1$ is connected, $d$ is a false twin of $v_{2}$ in $H_1$, and this contradicts the PD property of $H_1$. Therefore, $c , u_{2} \in N_x$ and $u_{2} \in H_1$. Since $N_G(v_{2})=N_G(v_1)-\{u_1\}$ and $u_{2}$ is adjacent to $v_{1}$, therefore $u_{2}$ is adjacent to $v_1$. So, $c$ is adjacent to $v_1$ and not to $v_{2}$. Consequently, $c=u_1$ and $N_G(u_1)=N_G(u_{2})-\{v_{2}\}$.
	
Generally, for $i\geq 2$, suppose $\mathcal{A}_i$ is an induced subgraph of $H_1$, where $V(\mathcal{A}_i)=\{u_1, \ldots, u_i , v_1 , \ldots, v_i\}$ and
	
	\begin{itemize}
		\item[(i)] $N_{G}(v_k )=N_{G}(v_{k-1} )-\{u_{k-1}\}$, for $2\leq k\leq i$,
		\item[(ii)] $N_{G}(u_k )=N_{G}(u_{k+1} )-\{v_{k+1}\}$, for $1 \leq k \leq i-1$,
		\item[(iii)] $\deg_{H_1}(u_k)=k$, for $1 \leq k \leq i$.
	\end{itemize}
	
	If $u_{i}\in H_1^\circ$, we have $u_{i} \in H^\circ$ by Lemma \ref{notejtema}. By ($*$), $u_{i}$ is adjacent to a pendant vertex in $H_1$. We have $N_{H_1}(u_{i})=\{v_1 , \dots, v_{i}\}$ and $\{u_l, \dots , u_{i}\} \subseteq N_{H_1}(v_l)$ for $1 \leq l\leq i$. Then the only valid candidate for this pendant vertex is $v_{i}$. In summation, $v_{i}$ is a pendant vertex adjacent to $u_{i}$ and $H_1 =\mathcal{A}_{i}$.
	
		If $u_i \notin H_1^\circ$, the subgraph $G-u_i$ is not PD. Then there exist vertices $ v_{i+1},v \in G-u_i$ such that $N_G(v_{i+1})=N_G(v)-\{u_i\}$. If both $v$ and $ v_{i+1}$ belong to $N_y$, then $v_{i+1}\neq x$. Therefore, by Lemma \ref{birooni}, there exists a vertex $w \in G-\{v_{i+1} , u_i\}$ such that $N_G(u_i)=N_G(w)-\{v_{i+1}\}$.  The vertex $w$ is adjacent to $x$, because $u_i$ is; this means $w \in N_x$. By restricting the equality $N_G(u_i)=N_G(w)-\{v_{i+1}\}$ to the subgraph $N_x$, we get $N_{N_x}(u_i)=N_{N_x}(w)$. Since $H_1$ is connected, $w$ is a false twin of $u_i$ in $H_1$, and this contradicts the PD property of $H_1$. Therefore, $v_{i+1},v \in N_x$ and more specifically $v_{i+1},v \in H_1$. By (ii) and (iii), $N_{H_1}(u_i)=\{v_1 , \dots, v_i\}$. So, $v=v_j$ for a $1 \leq j \leq i$. We know that $N_{\mathcal{A}_i}(v_l)=\{u_l , \dots, u_i\}$, thus $v_{i+1}\neq v_l$, for $ 1\leq l\leq i$. Also, $N_{H_1}(u_l)=\{v_1 , \cdots, v_l\}$ for $ 1\leq l\leq i$. If $j\neq i$, then $N_G(v_{i+1})=N_G(v_j) -\{u_i\}$ for some $1\leq j<i$, $\{v_1 , \dots,v_j , v_{i+1}\} \subseteq N_{H_1}(u_j)$ and hence $|N_{H_1}(u_j)| > j$, a contradiction to (iii). Therefore, $v=v_i$ (statement (i) for $k=i+1$). Using $N_G(v_{i+1})=N_G(v_i)-\{u_i\}$ and (i), inductively, we have $N_G(v_{i+1})=N_G(v_1)-\{u_1, \ldots, u_i\}$. Thus, $v_{i+1} \notin \{u_1, \ldots, u_n\}$. Since every vertex adjacent to $v_{i+1}$ is also adjacent to $v_i$, no pendant vertex in $H_1$ is adjacent to $v_{i+1}$. Thus,  the vertex $v_{i+1}$ cannot be a nucleus vertex of $H_1$. Since $N_G(v_{i+1})=N_G(v_i)-\{u_i\}$ and $H_1 \neq \mathcal{A}_i$, then $v_{i+1} \neq y$. Therefore, $G-v_{i+1}$ is not PD and there exist $s , u_{i+1} \in G - v_{i+1}$ such that $N_G(s)=N_G(u_{i+1})-\{v_{i+1}\}$. If both $s$ and $u_{i+1}$ belong to $N_y$, then $s \neq x$ and by Lemma \ref{birooni} there exists a vertex $t \in G-\{s , v_{i+1}\}$ such that $N_G(v_{i+1})=N_G(t)-\{s\}$; so $t \in N_x$ and the restriction of the equality $N_G(v_{i+1})=N_G(t)-\{s\}$ to $N_x$ results in $N_{N_x}(v_{i+1})=N_{N_x}(t)$.  Since $H_1$ is connected, $t$ is a false twin of $v_{i+1}$ in $H_1$, and this contradicts the PD property of $H_1$. Therefore, $s , u_{i+1} \in N_x$ and $u_{i+1} \in H_1$. Since $N_G(v_{i+1})=N_G(v_i)-\{u_i\}$ and $u_{i+1}$ is adjacent to $v_{i+1}$, the vertex $u_{i+1}$ is adjacent to $v_i$. So, $s$ is adjacent to $v_i$ and not to $v_{i+1}$. Therefore, $s=u_i$. We have $N_G(u_i)=N_G(u_{i+1})-\{v_{i+1}\}$ (statement (ii) for $k=i$) and since $\deg_{H_1}(u_i)=i$ then $\deg_{H_1}(u_{i+1})=i+1$ (statement (iii) for $k=i+1$) and $\mathcal{A}_{i+1}$ is an induced subgraph of $H_1$.
		
		Continuing this way, we conclude that $H_1=\mathcal{A}_n$ for some $n\geq1$ and the result follows.  
\end{proof}
The proof of Lemma \ref{n1}, also implies the following two lemmas. First we explicitly rewrite the phrase ($*$) in the proof of Lemma  \ref{n1}.
\begin{lemma}\label{existence}
	Let $G$ be a connected point determining graph with $G^\circ = \{x,y\}$ and $N_x=\alpha K_1+H $ where $H$ is a point determining graph with no isolated vertex. For every vertex $u \in H^\circ$, there exists a pendant vertex $v \in H$ such that $N_G(y)=N_G(v)-\{u\}$. \hfill $\square$

\end{lemma}

\begin{lemma}\label{eqG}
 Let $G$ be a connected point determining graph with $G^\circ = \{x,y\}$ and $N_x=\alpha K_1+H $ where $H$ is a point determining graph with no isolated vertex. For a connected component $\mathcal{A}_n$ of $H$, the following properties hold.
\begin{itemize}
	\item[(i)] $N_{G}(v_i)=N_{G}(v_{i-1} )-\{u_{i-1}\}$, for $2\leq i\leq n$,
	\item[(ii)] $N_{G}(u_i )=N_{G}(u_{i+1} )-\{v_{i+1}\}$, for $1 \leq i \leq n-1$,
	\item[(iii)] $N_G(y)=N_G(v_i)-\{u_i , \ldots , u_n\}$, for $1 \leq i\leq n$,
	\item[(iv)] $N_G(y)=N_G(u_i)-\{v_1, \ldots, v_i\}$, for $1 \leq i\leq n$.
\end{itemize}
\end{lemma}

\begin{proof}
	Parts (i) and (ii) are exactly the same as parts (i) and (ii) in the proof of Lemma \ref{n1}. Let $v_i \in V_G(\mathcal{A}_n)$, for $1 \leq i \leq n$. The vertex $v_1$ is a nucleus vertex of $\mathcal{A}_n$, therefore by Lemma \ref{notejtema}, $v_1 \in H^\circ$. By Lemma \ref{existence}, $v_1$ is adjacent to a pendant vertex in $H$, which is $u_1$, and we have $N_G(y)=N_G(v_1)-\{u_1\}$. Using induction on $i$ in the statement (i), gives (iii). The statement (iv) follows similarly with induction on $i$ and (ii).
\end{proof}

\begin{corol}\label{daraje}
	Let $G$ be a connected point determining graph with $G^\circ = \{x,y\}$.
	\begin{itemize}
		\item[(i)] If $z \neq y$ is isolated in $N_x$, then $\deg_G(z) < \deg_G(y)$,
		\item[(ii)] If $z$ belongs to some nontrivial connected component $\mathcal{A}_n$ of $N_x$, then $\deg_G(z)=\deg_G(y)+\deg_{\mathcal{A}_{n}}(z)$.
	\end{itemize}
\end{corol}
\begin{proof}
	(i) Among all isolated vertices in $N_x-y$, consider $z$ with largest degree in $G$. The subgraph $G-z$ is not PD, therefore, there exist vertices $p,q \in G-z$ such that $N_G(p)=N_G(q)-\{z\}$. Since $z$ is an isolated vertex in $N_x$, the vertices $p$ and $q$ are not in $N_x$ and $p \neq x$. Thus, there exist $r \in G-\{p,z\}$ such that $N_G(z) = N_G(r)-\{p\}$. The vertex $r$ is adjacent to $x$, because $z$ is. Thus, $r \in N_x$.  By restricting the equality $N_G(z) = N_G(r)-\{p\}$ to the subgraph $N_x$, we have $N_{N_x}(z)=N_{N_x}(r)$ and this means $r$ is an isolated vertex of $N_x$. Since $z$ had the largest degree among the isolated vertices of $N_x$, then $r=y$ and $\deg_G(z)=\deg_G(y)-1<\deg_G(y)$.
	\\(ii) Follows from parts (iii) and (iv) of Lemma \ref{eqG}.
\end{proof}

\begin{corol}\label{maxdaraje}
	Let $G$ be a connected point determining graph with $G^\circ = \{x,y\}$. Then the following statements hold.
	\begin{itemize}
		\item[(i)] If both $N_x$ and $N_y$ have no edges, then $\Delta(G)=\max \{\deg_G(x), \deg_G(y)\}$.
		\item[(ii)] If either $N_x$ or $N_y$ has an edge, then $\Delta(G)= \deg_G(z)$ for a non-isolated vertex $z \in N_x+N_y$. \hfill $\square$
	\end{itemize}
\end{corol}

\begin{remark}\label{xyconnected}
	Let $G$ be a point determining graph. By Lemma \ref{notejtema}, if $G$ is disconnected with no trivial component, then $|G^\circ| \geq 4$. If $G^\circ=\{x,y\}$, then either $G$ is connected or exactly one of $x$ and $y$ is the unique isolated vertex of $G$. As we will see later, since we are studying the reconstruction problem, we can further assume that $G-x \simeq G-y$. So, in fact, $G$ has no isolated vertex and is connected. 
\end{remark}

\begin{remark}\label{no.An}
	Let $G$ be a point determining graph such that $G^\circ=\{x,y\}$ and $G-x \simeq G-y$. Then $\deg_G(x)=\deg_G(y)=:d$. Every nontrivial connected component of $N_x+N_y$, which is $\mathcal{A}_n$ for some $n$, has exactly two pendant vertices. By Corollary \ref{daraje}, only the pendant vertices in $N_x+N_y$ have degree $d+1$ in $G$ and only $x$ and $y$ have degree $d$ in $G$. Thus the occurrence of $d$ in the degree sequence of $G-x$ is even and is twice the number of nontrivial components of $N_x$. Since the degree sequences of $G-x$ and $G-y$ are identical, we have the same number of nontrivial components in $N_x$ and $N_y$.
\end{remark}

Suppose $G$ is a PD graph. For $q \notin G^\circ$, there exist vertices $p , r \in G-q$ such that

\begin{equation}\label{hamsaye}
N_G(r)=N_G(p)-\{q\}
\end{equation}

If $r$, $p$ and $q$ satisfy Equation (\ref{hamsaye}), we write $p \rightarrow q  \rightsquigarrow r$ or $r \rightsquigarrow q \rightarrow p$. In \cite{feder2008}, these three vertices $p,q,r$ are called a triple and an equivalence relation between triples is defined. In this step, we adopt the same approach as in \cite{feder2008} restricted to the case $|G^\circ|=2$ and give the necessary details and proofs corresponding to the recognizability of being PD.

\begin{defi}\label{sequence}
	Let $G$ be a point determining graph. An $\Omega$-sequence is a sequence $(a_0,a_1,\ldots, a_k)$, $k\geq 3$, of distinct vertices, except possibly $a_0=a_k$, such that either of the following occurs.
	\begin{itemize}
		\item[(I)] $a_0 \rightarrow a_1 \rightsquigarrow a_2 \rightarrow a_3 \rightsquigarrow \cdots
			(\rightsquigarrow/\rightarrow) a_k$
	
		\item[(II)] $a_0 \rightsquigarrow a_1 \rightarrow a_2 \rightsquigarrow a_3 \rightarrow \cdots (\rightsquigarrow/\rightarrow) a_k$
	
	\end{itemize} A saturated $\Omega$-sequence $(a_0,a_1,\ldots, a_k)$, is an $\Omega$-sequence such that $a_0, a_k \in G^\circ$.
\end{defi}

\begin{remark}\label{fasele}
	Let $G$ be a point determining graph and $(a_0,a_1,\ldots,a_k)$ an $\Omega$-sequence. It is obvious that $a_i$ is not a nucleus vertex for $ 1\leq i \leq k-1$. The neighborhoods in (I) and (II) are in the form  $N_G(a_{i+2})=N_G(a_i)-\{a_{i+1}\}$ or $N_G(a_{i})=N_G(a_{i+2})-\{a_{i+1}\}$, for $0 \leq i \leq k-2$. Consider neighborhoods with the form $N_G(a_{i+2})=N_G(a_i)-\{a_{i+1}\}$. 
	
	\noindent For $a_{i+2}$ and $a_{i+4}$ we have
	$$N_G(a_{i+2})=N_G(a_{i})-\{a_{i+1}\} \quad \text{and} \quad N_G(a_{i+4})=N_G(a_{i+2})-\{a_{i+3}\}.$$
	One can deduce that $N_G(a_{i+4})=N_G(a_i)-\{a_{i+1}, a_{i+3}\}$ and $\{a_{i+1}, a_{i+3}\} \subseteq N_G(a_i)$. 
	
	\noindent For $a_{i+6}$, we have $N_G(a_{i+6})=N_G(a_{i+4})-\{a_{i+5}\}$ from which it follows that, $$N_G(a_{i+6})=N_G(a_i)-\{a_{i+1}, a_{i+3} , a_{i+5}\}, \quad \{a_{i+1}, a_{i+3} , a_{i+5}\} \subseteq N_G(a_i)$$
	and 
	$$ N_G(a_{i+6})=N_G(a_{i+2})-\{a_{i+3} , a_{i+5}\}, \quad \{a_{i+3} , a_{i+5}\} \subseteq N_G(a_{i+2}).$$
	Inductively, for $0\leq i \leq k-2$ and $2 \leq i+2j \leq k$, we have 
	$$N_G(a_{i+2j})=N_G(a_{i}) - \{a_{i+1} , a_{i+3} , \ldots a_{i+2j-1}\}, \quad \{a_{i+1} , a_{i+3} , \ldots a_{i+2j-1}\} \subseteq N_G(a_{i}).$$
	Similarly, from $N_G(a_i)=N_G(a_{i+2})-\{a_{i+1}\}$, we have
	$$N_G(a_{i})=N_G(a_{i+2j}) - \{a_{i+1} , a_{i+3} , \ldots a_{i+2j-1}\}, \quad \{a_{i+1} , a_{i+3} , \ldots a_{i+2j-1}\} \subseteq N_G(a_{i+2j}).$$
\end{remark}

\begin{lemma}\label{sequence-existence}
	Let $G$ be a point determining graph. Then an $\Omega$-sequence in $G$ can be extended to a saturated $\Omega$-sequence. Consequently, each non-nucleus vertex belongs to a saturated $\Omega$-sequence.
\end{lemma}
\begin{proof}
	Let $(a_0,a_1,\ldots, a_k)$ be an $\Omega$-sequence. Assume that $a_k \not\in G^\circ$. Two cases occur.
	\begin{itemize}
		\item[Case 1.] The $\Omega$-sequence ends with $a_{k-2} \rightsquigarrow a_{k-1} \rightarrow a_k$, equivalently, $N_G(a_{k-2})= N_G(a_k)-\{a_{k-1}\}$. Therefore, by Lemma \ref{birooni}, there exists $a_{k+1} \in G-\{a_k , a_{k-1}\}$ such that $N_G(a_{k+1})=N_G(a_{k-1})-\{a_k\}$. If we show that $a_k\neq a_0$ and $a_{k+1} \notin \{a_1,\ldots,a_{k-2}\}$, then $(a_0,\ldots,a_k,a_{k+1})$ is again an $\Omega$-sequence. 
		
		If $k$ is even, by Remark \ref{fasele}, $N_G(a_0)=N_G(a_k)-\{a_{1} , a_{3} , \ldots, a_{k-1}\}$. 
		Thus $a_k \neq a_0$. Again, by Remark \ref{fasele}, $N_G(a_{k+1})=N_G(a_{i})-\{a_{i+1}, a_{i+3},\ldots,a_{k}\}$, for odd $i$, $0 \leq i \leq k$. Since the neighborhoods of $a_{k+1}$ and $a_i$ are different, we have $a_{k+1} \neq a_i$ for each odd $i$, $0 \leq i \leq k$. Since $\{a_2, a_4, \ldots, a_k\} \subseteq N_G(a_1)$, we have $a_{k+1} \neq a_i$ for even $i$, $2 \leq i \leq k$. Hence $a_{k+1} \notin \{a_1,\ldots,a_k\}$.
		
		If $k$ is odd, once more by Remark \ref{fasele}, $N_G(a_1)=N_G(a_k)-\{a_{2} , a_{4} , \ldots, a_{k-1}\}$. Since $a_1$ is adjacent to $a_0$, so $a_0$ and $a_k$ are adjacent. 
		Thus $a_k \neq a_0$ and a similar argument as in the case $k$ even, shows that $a_{k+1} \notin \{a_1,\ldots,a_k\}$. In addition, in this case, $a_{k+1} \neq a_0$.
		
		In either case, $(a_0,a_1, \ldots, a_k,a_{k+1})$ is an $\Omega$-sequence. 
		
		\item[Case 2.] The $\Omega$-sequence ends with $a_{k-2} \rightarrow a_{k-1} \rightsquigarrow a_k$, equivalently, we have $N_G(a_{k})= N_G(a_{k-2})-\{a_{k-1}\}$. Therefore, by Lemma \ref{birooni}, there exists $a_{k+1} \in G-\{a_k , a_{k-1}\}$ such that $N_G(a_{k-1})=N_G(a_{k+1})-\{a_k\}$. As in Case 1, $a_{k+1} \notin \{a_1,\ldots,a_k\}$. If we show that $a_k\neq a_0$, then $(a_0,\ldots,a_k,a_{k+1})$ is an $\Omega$-sequence.
		
		If $k$ is even, by Remark \ref{fasele}, $N_G(a_k)=N_G(a_0)-\{a_{1} , a_{3} , \ldots, a_{k-1}\}$. 
		Thus $a_k \neq a_0 $.
		
		Let $k$ be odd. If $a_{k+1}=a_0$, since $a_k$ and $a_{k+1}$ are adjacent, then $a_k$ and $a_0$ are adjacent, too. Thus, $a_k \neq a_0$. If $a_{k+1} \neq a_0$, from $N_G(a_{k-1})=N_G(a_{k+1})-\{a_k\}$ and $N_G(a_0)=N_G(a_{k-1})-\{a_{1} , a_{3} , \ldots, a_{k-2}\}$, and by Remark \ref{fasele}, we have $N_G(a_0)=N_G(a_{k+1})-\{a_{1} , a_{3} , \ldots, a_{k-1}\}$. Thus, $a_k \neq a_0$. 
		
		In either case, $(a_0,a_1\ldots a_k,a_{k+1})$ is an $\Omega$-sequence. 
	\end{itemize}

	This process ends with a nucleus vertex. We do the same at the other end $a_0$ to reach a nucleus vertex. Finally, each non-nucleus vertex belongs to a triple which is, per se, an $\Omega$-sequence, that can be extended to a saturated $\Omega$-sequence.
\end{proof}

\begin{lemma}\label{kodd}
	Let $G$ be a point determining graph whose point determining cards are mutually isomorphic. If $(a_0,a_1,\ldots, a_k)$ is a saturated $\Omega$-sequence in $G$, then the following statements hold.
	\begin{itemize}
		\item[(i)] $k$ is odd.
		\item[(ii)] If $a_0=a_k$, then the sequence starts as in (II).
		\item[(iii)] If $|G^\circ|=2$ and $a_0\neq a_k$, then the sequence starts as in (I).
	\end{itemize}
\end{lemma}
\begin{proof}
	(i) First assume the saturated $\Omega$-sequence $(a_0,a_1,\ldots, a_k)$ starts as in (I). Since $a_0,a_k \in G^\circ$, we have $G-a_0 \simeq G-a_k$. Therefore $\deg_G(a_0)=\deg_G(a_k)$. If $k$ is even, then by Remark \ref{fasele}, 
	$$N_G(a_k)=N_G(a_0)-\{a_1 , a_3, \ldots, a_{k-1}\},$$
	which is a contradiction.
	
	 The case for a saturated $\Omega$-sequence $(a_0,a_1,\ldots, a_k)$ that starts as in (II), is similar.
	
	\noindent(ii) In contrary, assume that $(a_0,a_1,\ldots, a_k)$ is a saturated $\Omega$-sequence that starts as in (I). By (i), $k$ is odd therefore, by Remark \ref{fasele}, 
	$$N_G(a_1)=N_G(a_{k})-\{a_2, a_4, \ldots a_{k-1}\}.$$
	 Since $a_1$ and $a_0$ are adjacent, then $a_0$ and $a_k$ are adjacent, too; a contradiction.

	\noindent(iii) In contrary, assume that $(a_0,a_1,\ldots, a_k)$ is a saturated $\Omega$-sequence that starts as in (II). Again by (i), $k$ is odd and
	$$N_G(a_0)=N_G(a_{k-1})-\{a_1, a_3 , \ldots , a_{k-2}\}.$$
	 Since $a_{k-1}$ and $a_k$ are non-adjacent, then $a_0$ and $a_k$ are non-adjacent, too. This contradicts Theorem \ref{xya}.
\end{proof}

\begin{defi}
	Let $G$ be a point determining graph such that its point determining cards are mutually isomorphic and $S=(a_0,a_1,\ldots, a_k)$ be a saturated $\Omega$-sequence in $G$. Then $S$ is called a belt, if it starts as in (II) (Definition \ref{sequence}) and $a_0=a_k$; is called an N-sequence if it starts as in (II) and $a_0 \neq a_k$; and is called an A-sequence if it starts as in (I) and $a_0 \neq a_k$.
\end{defi}

	Let $G$ be a point determining graph such that its point determining cards are mutually isomorphic. Let $S=(a_0,a_1,\ldots, a_k)$ be a saturated $\Omega$-sequence. If $S$ is an A-sequence, then by Remark \ref{fasele}, for $0 \leq i \leq k-2$, we have 
	$$N_G(a_i)=N_G(a_0)-\{a_1 , a_3 , \ldots , a_{i-1}\},$$
	if $i$ is even, and
	$$N_G(a_i)=N_G(a_k)-\{a_{i+1} , a_{i+3}, \ldots , a_{k-1}\},$$
	if $i$ is odd.
	Consequently, $\deg_{G}(a_i)< \deg_{G}(a_0)=\deg_{G}(a_k)$ and if $|G^\circ|=2$, by Remark \ref{xyconnected} and Corollary \ref{daraje}, $a_i$ is an isolated vertex in $N_x+N_y$, for $ 0 < i < k$.

	Now let $S$ be a belt or an N-sequence. By Remark \ref{fasele}, for $0 \leq i \leq k-2$, we have 
	$$N_G(a_0)=N_G(a_i)-\{a_1 , a_3 , \ldots , a_{i-1}\},$$
	if $i$ is even, and
	$$N_G(a_k)=N_G(a_i)-\{a_{i+1} , a_{i+3}, \ldots , a_{k-1}\},$$
	if $i$ is odd.
	Consequently, $\deg_{G}(a_i)> \deg_{G}(a_0)=\deg_{G}(a_k)$ and if $|G^\circ|=2$, by Remark \ref{xyconnected} and Corollary \ref{daraje}, $a_i$ is not an isolated vertex in $N_x+N_y$, for $ 0 < i < k$.

So, we deduce the following result.
\begin{lemma}\label{asequence}
	Let $G$ be a point determining graph with $G^\circ=\{x,y\}$ such that $G-x \simeq G-y$. Then a non-nucleus vertex of $G$ belongs to an A-sequence if and only if it is an isolated vertex in $N_x+N_y$. \hfill$\square$
\end{lemma}

\begin{lemma}\label{unique-sequence}
	Let $G$ be a point determining graph with $G^\circ=\{x,y\}$ such that $G-x\simeq G-y$. If every saturated $\Omega$-sequence of $G$ is an A-sequence, then each non-nucleus vertex belongs to exactly one A-sequence. 
\end{lemma}

\begin{proof}
By Remark \ref{xyconnected}, $G$ is connected. By Lemma \ref{sequence-existence}, every non-nucleus vertex of $G$ belongs to at least one saturated $\Omega$-sequence. Since the two PD cards of $G$ are isomorphic, $\deg_G(x)=\deg_G(y)$.
 Suppose, to the contrary, that there exists a non-nucleus vertex $u$ that lies in two A-sequences $(a_0,a_1,\ldots, a_k)$ and $(b_0,b_1,\ldots, b_{k'})$. Suppose that $a_0=b_0=x$ and $a_k=b_{k'}=y$ and $u=a_i=b_j$ for some $i \neq 0,k$ and $j \neq 0,k'$. By Remark \ref{fasele}, if $i$ is even, then $\deg_G(a_0)=\deg_G(a_i) + \frac{i}{2}$ and if $i$ is odd, then $\deg_{G}(a_i)=\deg_{G}(a_k)+ \frac{i-1}{2}$. Similarly, if $j$ is even, $\deg_G(b_0)=\deg_G(b_j) + \frac{j}{2}$ and if it is odd, $\deg_{G}(b_j)=\deg_{G}(b_{k'})+ \frac{j-1}{2}$. Therefore, $|i-j|\leq 1$. By Remark \ref{fasele}, the vertices in both the sequences $(a_0,a_1,\ldots, a_k)$ and $(b_0,b_1,\ldots, b_{k'})$ with odd indices are adjacent to $a_0$ and the vertices with even indices are adjacent to $a_k$. By Theorem \ref{xya}, each non-nucleus vertex is adjacent to exactly one nucleus vertex. Therefore,
 \begin{itemize}
 	\item[($*$)] If $a_i=b_j$ for $i \neq k$ and $j \neq k'$, then  $i=j$.
 \end{itemize}
 First suppose $i=2i'$ is even, then
 $$N_G(a_0)-\{a_1, a_3 , \ldots, a_{2i'-1}\}=N_G(a_i)=N_G(b_j)=N_G(b_0)-\{b_1, b_3, \ldots , b_{2i'-1}\}.$$
 Therefore, $\{a_1, a_3 , \ldots, a_{2i'-1}\} = \{b_1, b_3, \ldots , b_{2i'-1}\}$. Now by ($*$) we have $a_1=b_1, a_3=b_3,\ldots , a_{2i'-1}=b_{2i'-1}$. Since $a_1=b_1$, then 
 $$N_G(a_k)-\{a_2 , a_4 , \cdots , a_{k-1}\} = N_G(a_1)= N_G(b_{1})= N_G(b_{k'})-\{b_2 , b_4 , \cdots , b_{k'-1}\}. $$
 Therefore, $\{a_2 , a_4 , \cdots , a_{k-1}\} = \{b_2 , b_4 , \cdots , b_{k'-1}\}$. Thus, $k=k'$ and again by ($*$), we have $a_2=b_2, a_4=b_4,\ldots , a_k=b_k$. Similarly, with $a_{k-1}=b_{k-1}$ and ($*$), we have $a_1=b_1,a_3=b_3,\ldots , a_{k-2}=b_{k-2}$. Therefore $(a_0, \ldots , a_k)=(b_0,\ldots,b_k)$.
 
 If $i$ is odd, then consider the A-sequences $(a_k,a_{k-1},\ldots,a_0)$ and $(b_{k'},b_{k'-1},\ldots,b_0)$ and do the same reasoning to conclude that $k=k'$ and $(a_0,\ldots,a_k)=(b_0,\ldots,b_k)$.
\end{proof}

\begin{lemma}\label{xy}
	Let $G$ be a point determining graph with $G^\circ=\{x,y\}$ such that $G-x \simeq G-y$. If $N_x = \alpha K_1$ for some $\alpha \geq 1$, then there exists an automorphism $\varphi :G \rightarrow G$ that maps $x$ to $y$.
\end{lemma} 
\begin{proof}
	  By Remark \ref{xyconnected}, $G$ is connected. If $G=K_2$, there is nothing to prove. So assume that $G\neq K_2$ and there exists a non-nucleus vertex in $G$. Since $N_x=\alpha K_1$, by Remark \ref{no.An}, $N_y=\alpha K_1$, too. By Lemmas \ref{asequence} and \ref{unique-sequence}, each non-nucleus vertex of $G$ belongs to exactly one A-sequence. 
	  For every A-sequence $S=(a_0,\ldots,a_k)$, define $\varphi:S \to S$ to be the reflection with respect to its virtual mid-point of $a_{\frac{k-1}{2}} a_{\frac{k+1}{2}}$ (either the cases $a_{\frac{k-1}{2}} \rightarrow a_{\frac{k+1}{2}}$ or $a_{\frac{k-1}{2}} \rightsquigarrow a_{\frac{k+1}{2}}$ may occur, depending on $k$); that is $\varphi (a_j)=a_{k-j}$, for $0\leq j \leq k$. Then $\varphi:G\to G$ is well-defined, because the endpoints of each A-sequence are $x$ and $y$, and every non-nucleus vertex belongs to exactly one A-sequence. The mapping $\varphi$ is bijective. We show that $\varphi$ is an automorphism. First let $u, v \in V(G)$ be adjacent. Four cases occur.
	  \begin{itemize}
	  	\item[Case 1.] $u,v \in \{x,y\}$. Any A-sequence has $x$ and $y$ as its endpoints and by definition, $\varphi(x)=y$ and $\varphi(y)=x$.
	  	\item[Case 2.] $u \not\in \{x,y\}$ and $v\in \{x,y\}$. In any A-sequence, the vertices with odd indices are adjacent to the initial vertex and the vertices with even indices are adjacent to the final vertex. Therefore, the vertices adjacent to $x$, map to vertices adjacent to $y$, and vice versa. 
	  	\item[Case 3.] $u, v \not\in \{x,y\}$ and they belong to the same A-sequence $S=(a_0,a_1,\ldots , a_k)$. Therefore there exist $ 1 \leq i < j \leq k-1$ such that $u=a_i$ and $v=a_j$. By Remark \ref{fasele}, adjacent vertices in $S$ have different parities and $i$ is even. We have $\varphi(a_i)= a_{k-i}$ and $\varphi(a_j)= a_{k-j}$. Hence, $k-j \leq k-i$ and $k-j$ is even. Again by Remark \ref{fasele}, $a_{k-i}$ and $a_{k-j}$ are adjacent.
	  	\item[Case 4.]  $u,v \not\in \{x,y\}$ and they belong to different A-sequences. Let $u \in S= (x, a_1,a_2 , \ldots, a_{k-1},y)$ and $v \in T= (x, b_1,b_2 , \ldots, b_{k'-1},y)$. Without loss of generality, assume that $u= a_i$ and $i$ is even (if $i$ is not even, look at the inverted sequences of $S$ and $T$). Therefore, by Remark \ref{fasele}, $N_G(a_i)=N_G(x)-\{a_1 , a_3 , \ldots , a_{i-1}\}$. Then, $v$ is adjacent to $x$, since $v$ is adjacent to $a_i$. Hence, $v=b_j$ for some odd $j$, $1 \leq j \leq  k'-1$. Also,  $k-i$ is odd and we have $N_G(a_{k-i}) = N_G(y) -\{a_{k-i+1}, a_{k-i+3}, \ldots a_{k-1}\}$.  By Case 2, $b_j$ is adjacent to $x$, therefore, $\varphi(b_j)=b_{k'-j}$ is adjacent to $y$. Thus, $\varphi (a_i)$ is adjacent to $\varphi(b_j)$.
	  \end{itemize}
	  The cases where $u$ and $v$ are not adjacent are similar, mutatis mutandis. Therefore, $\varphi$ is an automorphism.
\end{proof}

\begin{theorem}\label{xykoli}
	Let $G$ be a point determining graph with $G^\circ=\{x,y\}$ and $G-x \simeq G-y$, then there exists an automorphism $\varphi :G \rightarrow G$ that maps $x$ to $y$.
\end{theorem}
\begin{proof}
	By Remark \ref{xyconnected}, $G$ is connected. Let $N_x=\alpha K_1+H$ where $H$ has no isolated vertices. By Lemma \ref{nKH}, $\alpha \geq 1$ and $H$ is PD. If $H=\emptyset$, i.e. $N_x= \alpha K_1$, then by Lemma \ref{xy}, the automorphism $\varphi$ with the desired property $\varphi(x)=y$ exists.
	
	Now let $H\neq \emptyset$ and $\varphi' : G-x \rightarrow G-y $ be an isomorphism. We show that $\varphi'(y)=x$. If $\varphi'(y)=z \neq x$, then $\deg_{G-y}(z)= \deg_G(y)-1$. By Corollary \ref{daraje}, the only valid candidates for $z=\varphi'(y)$ are the isolated vertices in $N_x$. For an isolated vertex $v$ in $N_y$, we have $\deg_G(v) < \deg_G(x)=\deg_G(y)$, therefore, $\deg_{G-y}(v) < \deg_{G}(y)-1$. The graph $G-z$ is not PD, hence there exist vertices $p,q \in G-z$ such that $N_G(p)=N_G(q)-\{z\}$. The vertex $z$ is isolated in $N_x$ and $z \neq y$, thus $q \in N_y$ and $q \neq y$. Since $q$ is adjacent to $y$, we have $p \neq y$ and hence $N_{G-y}(p)=N_{G-y}(q)-\{z\}$.
	
	The vertex $z$ is not nucleus in $G-y$, therefore $y \notin (G-x)^\circ$ and there exist vertices $r,s \in G-x-y$ such that $N_{G-x}(r) = N_{G-x}(s) -\{y\}$. By Theorem \ref{xya}, $r$ and $s$ are adjacent to exactly one of $x$ and $y$. Therefore, $N_G(r)-\{x\}=N_G(s)-\{y\}$.
	
	Assume that $w \in N_G(r)-\{x\}=N_G(s)-\{y\}$. By Theorem \ref{xya}, $w$ is in exactly one of $N_x$ or $N_y$. Without loss of generality suppose that $w \in N_x$. Since $w,r \in N_x$, the vertex $w$ is not isolated in $N_x$ and by Lemma \ref{n1} there exists some component $\mathcal{A}_n$ of $N_x$ containing $w$ and $r$. By Lemma \ref{eqG}, $N_G(y) \subseteq N_G(r)$, therefore $s$ is adjacent to $r$ and this contradicts the equality $N_G(r)-\{x\}=N_G(s)-\{y\}$. Hence, $N_G(r)-\{x\}=N_G(s)-\{y\}=\emptyset$ and the vertices $r$ and $s$ are pendant vertices in $G$ adjacent respectively to $x$ and $y$. The only possible case that $r$ and $s$ are pendant vertices in $G$ is that $N_x$ and $N_y$ have only isolated vertices. This contradicts the assumption $H \neq \emptyset$. Therefore $\varphi'(y)=x$.

	Now we extend $\varphi'$ to $\varphi : G\rightarrow G$ by $\varphi(x)=y$ and show that $\varphi$ is an automorphism. Clearly $\varphi$ is bijective. For a vertex $t \in G-x$, the vertex $y$ is adjacent to $t$ if and only if $x$ is adjacent to $\varphi'(t)$. Therefore $\varphi'(N_{G-x}(y))= N_{G-y}(x)$. For $t \not\in N_{G-x}(y) $, i.e $t \in N_x$, we have $\varphi'(t) \in N_y$. To conclude that $\varphi$ is an automorphism, it suffices to show that for any vertex $v$ in $G$, $v$ is adjacent to $x$ if and only if $\varphi(v)$ is adjacent to $y=\varphi(x)$.

	If $v$ and $x$ are adjacent, then $v \in N_x$ and hence $\varphi(v) \in N_y$. This means $\varphi(v)$ and $\varphi(x)=y$ are adjacent. If $v$ and $x$ are not adjacent, then $v \in N_y$ and hence $\varphi(v) \in N_x$. This means $\varphi(v)$ and $\varphi(x)=y$ are not adjacent.
\end{proof}

\section{Recognizability of Being Point Determinating}\label{main}

In this section, we show that the property of being point determining is recognizable.

\begin{lemma}\label{2pd}
	Let $G$ be a graph that is not point determining. Then either exactly two cards of $G$ are point determining or none of them are. Furthermore, if $G$ has exactly two point determining cards, then they are isomorphic and their corresponding deleted vertices are the only pair of false twins in $G$.
\end{lemma}
\begin{proof}
	The graph $G$ is not point determining, therefore there exist non-adjacent vertices $x,y \in V(G)$ such that $N_G(x)=N_G(y)$. For every $z \in V(G)-\{x,y\}$, since $N_G(x)=N_G(y)$, $z$ is either adjacent to both $x$ and $y$ or to none of them. Therefore, $x$ and $y$ are false twins in the cards $G-z$ and these cards are not PD. 

	First suppose that there are false twins $u,v \in V(G)$ such that $\{u,v\} \neq \{x,y\}$. If $\{u,v\} \cap \{x,y\} \neq \emptyset$, without loss of generality, assume that $u=x$. Then $N_G(y)=N_G(x)=N_G(u)=N_G(v)$. In other words, $v$ and $y$ are also false twins in $G$. Therefore, $u$ and $v$ are false twins of $G-y$ and the vertices $v$ and $y$ are false twins of $G-x$. If $\{u,v\} \cap \{x,y\} = \emptyset$, then $u$ and $v$ are false twins in both $G-x$ and $G-y$. In either case, $G$ has no PD cards.
	
	Second suppose that $x$ and $y$ are the only false twins of $G$. Then we show that $G$ has only $G-x$ and $G-y$ as its PD cards. Assume $p$ and $q$ are false twins in $G-x$.  Since $x$ and $y$ are the only false twins of $G$, exactly one of $p$ or $q$, say $p$, is adjacent to $x$. Hence, $y$ is adjacent to $p$ and not to $q$. Thus, $N_{G-x}(q)=N_{G-x}(p)-\{y\}$, which contradicts the equality $N_{G-x}(q)=N_{G-x}(p)$. Consequently, the card $G-x$ is PD, and the case for $G-y$ is similar.
	
	Now suppose that $G$ has exactly two point determining cards. Since $G$ is not PD, there exist false twins $a$ and $b$ in $G$. We showed that, the only cards that are possibly PD, are $G-a$ and $G-b$. Thus, these are the point determining cards of $G$ and $a$ and $b$ are the only pair of false twins of $G$. Finally, since $N_G(a)=N_G(b)$, the cards $G-a$ and $G-b$ are isomorphic.
\end{proof}
In other words, Lemma \ref{2pd} says that if a deck of cards has just one or at least three PD cards, then the solution graph is PD. Furthermore, by Theorem \ref{nucleus}, if a deck has no PD cards, the solution graph is not PD. 

Let $G$ be a graph and $x \in V(G)$. A blow-up of $x$ in $G$, is adding a new vertex $x'$ to $G$ with the same neighborhood as $x$. We denote the resulting graph by $G^+(x)$. Clearly, $x'$ is a false twin of $x$ in the new graph. The converse is also true, if $x$ and $x'$ are false twin in $G$, then $G=(G-x')^+(x)$.

\begin{lemma} \label{deck2pd}
Let $G$ be a graph with exactly two PD cards, $G-x$ and $G-y$. If $G-x \simeq G-y$, then $G$ is not point determining if and only if there exists a vertex $v \in G-x$ with $\deg_G(x)=\deg_{G-x}(v)$ such that $\Delta(G)=\Delta((G-x)^+(v))$.
\end{lemma}

\begin{proof}
First assume that $G$ is not point determining. By Lemma \ref{2pd}, $G$ has exactly one pair of false twins. The cards of these false twins are PD cards. So $x$ and $y$ are the unique pair of false twins of $G$. Since $N_G(x)=N_G(y)$, then $\deg_G(x)=\deg_G(y)$. The graph $(G-x)^+(y)$ is isomorphic to $G$. Therefore, $\Delta(G)=\Delta((G-x)^+(y))$.

Second assume that $G$ is point determining, then $G^\circ=\{x,y\}$. By Corollary \ref{daraje}, the only vertices $z$ such that $\deg_{G-x}(z)=\deg_G(x)$ are the pendant vertices in $N_x$. If $N_x=\alpha K_1$ for some $\alpha \geq 1$, then the lemma holds. If $N_x$ has an edge, let $v$ be a pendant vertex in $N_x$. By Lemma \ref{existence}, $N_G(y)=N_G(v)-\{u\}$ for some $u \in N_x$ ($u$ is a nucleus vertex in non-empty PD component of $N_x$). Since $G-x \simeq G-y$, by Theorem \ref{xykoli}, $N_x \simeq N_y$. By Corollary \ref{maxdaraje}, there exists $n \geq 1$ such that $\mathcal{A}_n$ (according to its definition with $V(\mathcal{A}_n)=\{u_1,\ldots,u_n,v_1,\ldots,v_n\}$) is the largest connected PD component of $N_y$ and $\Delta(G)=\deg_G(u_n)= n+ \deg_G(x)$. By (iv) of Lemma \ref{eqG}, $N_G(x)=N_G(u_n)-\{v_1 , \ldots, v_n\}$. Therefore, $v$ and $u_n$ are adjacent. Since $x$ and $u_n$ are not adjacent, $\deg_G(u_n)=\deg_{G-x}(u_n)$. Consider $H=(G-x)^+(v)$. Since $v$ and $u_n$ are adjacent in $G$, the new vertex in $H$ is also adjacent to $u_n$. Therefore, $\deg_H(u_n)=\deg_G(u_n)+1$, in other words, $\Delta(H)=\Delta(G)+1$.
\end{proof}

Now all ingredients are ready to prove the main theorem.

\begin{theorem}\label{maintheorem}
	Point determining property is recognizable for graphs of order at least three.
\end{theorem}
\begin{proof}
	Let $G$ be a graph of order at least $3$. Let $\mu$ be the number of PD cards of $G$. If $\mu =1$ or $\mu \geq 3$, then by Lemma \ref{2pd}, $G$ is a PD graph with $|G^\circ|=\mu$. If $\mu =0$, then by Theorem \ref{nucleus}, $G$ is a non-PD graph. 
	So, let $\mu =2$ and let $G-a$ and $G-b$ be the PD cards of $G$. By Lemma \ref{2pd}, if $G-a \not\simeq G-b$, then $G$ is PD with $|G^\circ|=2$.
	The last case is when $G-a \simeq G-b$. By Lemma \ref{deck2pd}, if there exists $v \in G-a$ such that $\deg_{G-a}(v)=\deg_G(a)$ and $\Delta((G-a)^+(v))=\Delta(G)$, then $G$ is non-PD; otherwise $G$ is PD with $|G^\circ|=2$.
\end{proof}

Proof of Theorem \ref{maintheorem}, is a procedure to establish whether a graph is point determining by looking at its deck.

\begin{example}\label{mesalpd}
	In Figure \ref{shklpdxy}, the maximum vertex degree of the solution is $4$ and the deck has exactly two PD cards which are isomorphic and their corresponding deleted vertices have degree $3$. Blowing-up any vertex of degree $3$ in these PD cards results in a graph which its maximum degree is strictly greater than $4$. This implies that the solution graph is PD.
\end{example}

\begin{figure}[h] 
	\begin{center}
		\begin{tikzpicture}[
			vertex/.style={
				circle,
				draw,
				fill=white,
				inner sep=1.5pt,
				minimum size=6pt
			},
			edge/.style={line width=0.9pt},
			curved/.style={
				line width=0.9pt
			},
			separator/.style={
				line width=0.7pt,
				dash pattern=on 3pt off 2pt on 1pt off 2pt
			}]
			
			
			\node[vertex] (a1) at (-0.45,0.225) {};
			\node[vertex] (a2) at (0.65,0.725) {};
			\node[vertex] (a3) at (0.65,-0.275) {};
			\node[vertex] (a4) at (2.05,0.725) {};
			\node[vertex] (a5) at (2.05,-0.275) {};
			
			\draw[edge] (a1)--(a2);
			\draw[edge] (a1)--(a3);
			
			\draw[edge] (a2)--(a3);
			\draw[edge] (a2)--(a4);
			\draw[edge] (a3)--(a5);
			
			\draw[edge] (a2)--(a5);
			\draw[edge] (a3)--(a4);
			
			\draw[edge] (a4)--(a5);

			
			\node[vertex] (b1) at (3.7,0.5) {};
			\node[vertex] (b2) at (4.9,0) {};
			\node[vertex] (b3) at (4.9,1) {};
			\node[vertex] (b4) at (6.2,1) {};
			\node[vertex] (b5) at (7.5,0.5) {};
			
			\draw[edge] (b1)--(b2);
			\draw[edge] (b1)--(b3);
			\draw[edge] (b2)--(b3);
			\draw[edge] (b3)--(b4);
			\draw[edge] (b2)--(b4);
			\draw[edge] (b4)--(b5);
			
			\draw[curved]
			(b1) to[out=-70,in=-110] (b5);
			
			\node[vertex] (c1) at (8.5,0.5) {};
			\node[vertex] (c2) at (9.7,0) {};
			\node[vertex] (c3) at (9.7,1) {};
			\node[vertex] (c4) at (11,1) {};
			\node[vertex] (c5) at (12.3,0.5) {};
			
			\draw[edge] (c1)--(c2);
			\draw[edge] (c1)--(c3);
			\draw[edge] (c2)--(c3);
			\draw[edge] (c3)--(c4);
			\draw[edge] (c2)--(c4);
			\draw[edge] (c4)--(c5);

			\draw[curved]
			(c1) to[out=-70,in=-110] (c5);
			
			
			\node[vertex] (d1) at (-0.45,-2.225) {};
			\node[vertex] (d2) at (0.65,-1.725) {};
			\node[vertex] (d3) at (0.65,-2.725) {};
			\node[vertex] (d4) at (2.05,-1.725) {};
			\node[vertex] (d5) at (2.05,-2.725) {};
			
			\draw[edge] (d1)--(d2);
			\draw[edge] (d1)--(d3);
			
			\draw[edge] (d2)--(d3);
			\draw[edge] (d2)--(d4);
			\draw[edge] (d3)--(d5);
			
			\draw[edge] (d2)--(d5);
			\draw[edge] (d3)--(d4);
			
			\draw[edge] (d4)--(d5);
			
			
			\node[vertex] (e1) at (3.7,-2) {};
			\node[vertex] (e2) at (4.9,-1.5) {};
			\node[vertex] (e3) at (4.9,-2.5) {};
			\node[vertex] (e4) at (6.2,-1.5) {};
			\node[vertex] (e5) at (7.5,-2) {};
			
			\draw[edge] (e1)--(e2);
			\draw[edge] (e1)--(e3);
			\draw[edge] (e2)--(e3);
			
			\draw[edge] (e2)--(e4);
			\draw[edge] (e3)--(e4);
			
			\draw[edge] (e4)--(e5);
			
			\draw[curved]
			(e1) to[out=-70,in=-110] (e5);
			
			
			\node[vertex] (f1) at (8.5,-2) {};
			\node[vertex] (f2) at (9.7,-1.5) {};
			\node[vertex] (f3) at (9.7,-2.5) {};
			\node[vertex] (f4) at (11.1,-1.5) {};
			\node[vertex] (f5) at (12.3,-2) {};
			
			\draw[edge] (f1)--(f2);
			\draw[edge] (f1)--(f3);
			\draw[edge] (f2)--(f3);
			
			\draw[edge] (f2)--(f4);
			\draw[edge] (f3)--(f4);
			
			\draw[edge] (f4)--(f5);
			
			\draw[curved]
			(f1) to[out=-70,in=-110] (f5);
			
			
			\draw[separator] (3.2,-3.25)--(3.2,1.25);
			\draw[separator] (8,-3.25)--(8,1.25);
			\draw[separator] (-1.6,-1)--(12.8,-1);
			
		\end{tikzpicture}
		\caption{A deck with exactly two PD cards (Example \ref{mesalpd})}\label{shklpdxy}
	\end{center}
\end{figure}
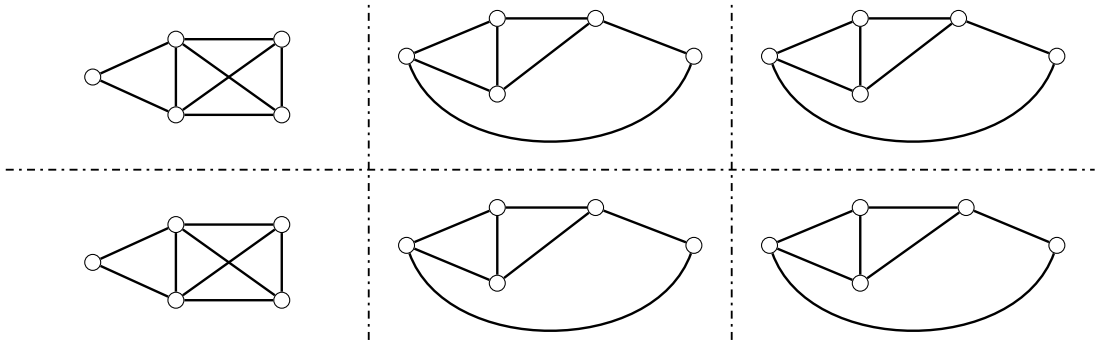
\FloatBarrier

\begin{example} \label{mesalnonpd}
	In contrast to Example \ref{mesalpd}, in the graph depicted in Figure \ref{shklnonpdxy}, the maximum degree is $4$ and it has only two PD isomorphic cards which their corresponding deleted vertices have degree 1. Blowing up the pendant vertex in these PD cards, results in a graph with the same maximum degree as the solution graph. Consequently, the solution graph is not PD.
\end{example}

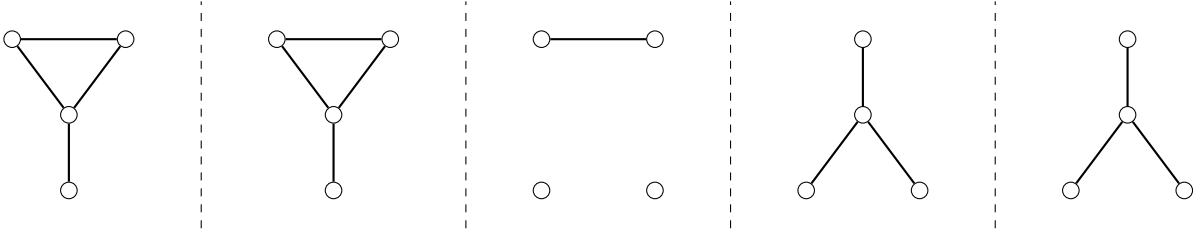
\begin{figure}[ht]
	\begin{center}
		\begin{tikzpicture}[
			vertex/.style={circle, draw, fill=white, inner sep=2.2pt},
			line/.style={thick}
			]
			
			\node[vertex] (a1) at (0,2) {};
			\node[vertex] (a2) at (1.5,2) {};
			\node[vertex] (a3) at (.75,1) {};
			\node[vertex] (a4) at (0.75,0) {};
			
			\draw[line] (a1)--(a2);
			\draw[line] (a1)--(a3);
			\draw[line] (a2)--(a3);
			\draw[line] (a3)--(a4);
			
			\node[vertex] (b1) at (3.5,2) {};
			\node[vertex] (b2) at (5,2) {};
			\node[vertex] (b3) at (4.25,1) {};
			\node[vertex] (b4) at (4.25,0) {};
			
			\draw[line] (b1)--(b2);
			\draw[line] (b1)--(b3);
			\draw[line] (b2)--(b3);
			\draw[line] (b3)--(b4);
			
			\node[vertex] (c1) at (7,2) {};
			\node[vertex] (c2) at (8.5,2) {};
			\node[vertex] (c3) at (7,0) {};
			\node[vertex] (c4) at (8.5,0) {};
			
			\draw[line] (c1)--(c2);
			
			\node[vertex] (d1) at (11.25,2) {};
			\node[vertex] (d2) at (11.25,1) {};
			\node[vertex] (d3) at (10.5,0) {};
			\node[vertex] (d4) at (12,0) {};
			
			\draw[line] (d1)--(d2);
			\draw[line] (d2)--(d3);
			\draw[line] (d2)--(d4);
			
			\node[vertex] (e1) at (14.75,2) {};
			\node[vertex] (e2) at (14.75,1) {};
			\node[vertex] (e3) at (14,0) {};
			\node[vertex] (e4) at (15.5,0) {};
			
			\draw[line] (e1)--(e2);
			\draw[line] (e2)--(e3);
			\draw[line] (e2)--(e4);
			
			\draw[dashed] (2.5,-.5)--(2.5,2.5);
			\draw[dashed] (6,-.5)--(6,2.5);
			\draw[dashed] (9.5,-.5)--(9.5,2.5);
			\draw[dashed] (13,-.5)--(13,2.5);
			
		\end{tikzpicture}
		\caption{A deck with exactly two PD cards (Example \ref{mesalnonpd})}\label{shklnonpdxy}
	\end{center}
\end{figure}

\section{Conclusion} \label{conclusion}
The following is a direct consequence of the characterization of point determining graphs with $|G^\circ|=1$ in Section \ref{G01}.

 \begin{theorem}
 	Let $G$ be a graph that has exactly one point determining card $H$. Then $G=H+K_1$. \hfill $\square$
 \end{theorem}
 
 Although this theorem is not novel, as it can already be derived from the disconnectedness of the graph, its significance lies in the method of proof rather than in the result.
 
 Our approach motivates the following question.
 
 \begin{que}
 	Let $G$ be a point determining graph such that each two point determining cards of $G$ are isomorphic. Can we reconstruct $G$ by its point determining cards?
 \end{que}

 The number of identical PD cards is important. 
 \begin{example}
 	The point determining cards of $C_5$ and $P_4 +K_1$ are isomorphic to $P_4$, but $C_5 \not\simeq P_4+K_1$. The graph $P_4+K_1$ has only one point determining card, while $C_5$ has five.
 \end{example}
 
 The point determining property of $G$ is also important.
 
\begin{example}
	Consider the path $P_5$ with consecutive vertices $x_1,x_2,x_3,x_4,x_5$. The graphs $P_5^+(x_2)$ and $P_5^+(x_3)$ are not point determining and each has exactly two point determining cards isomorphic to $P_5$, but $P_5^+(x_2) \not\simeq P_5^+(x_3)$.
\end{example} 
 
\bibliographystyle{plain}

\end{document}